\documentclass[final,leqno]{siamltex704}
\usepackage{amsmath}
\usepackage{amssymb}
\usepackage{graphicx}
\usepackage[notcite,notref]{showkeys}
\usepackage{tikz}
\usepackage{float}
\newcounter{dummy} \numberwithin{dummy}{section}
\newtheorem{defi}[dummy]{Definition}

\newtheorem{algorithm}{Algorithm}

\usepackage{calligra}
\DeclareMathAlphabet{\mathpzc}{OT1}{pzc}{m}{it}

\newcommand{\bt}{{\bf t}}
\newcommand{\bu}{{\bf u}}

\newcommand{\bw}{{\bf w}}

\newcommand{\bv}{{\bf v}}

\def\T{{\mathcal T}}
\def\E{{\mathcal E}}
\def\F{{\mathcal F}}

\def\pT{{\partial T}}
\def\l{{\langle}}
\def\r{{\rangle}}

\def\T{{\mathcal T}}

\def\bbf{{\bf f}}
\def\bg{{\bf g}}
\def\bn{{\bf n}}

\def\bbQ{\mathbb{Q}}

\def\3bar{{|\hspace{-.02in}|\hspace{-.02in}|}}

\title{ A Discrete Divergence Free Weak Galerkin Finite Element method
 for the Stokes Equations}
\author{Lin Mu\thanks{Computer Science and Mathematics Division
Oak Ridge National Laboratory, Oak Ridge, TN, 37831,USA
(mul1@ornl.gov). This research was supported in part by the U.S.~Department of Energy, Office of Science, Office of Advanced Scientific Computing Research, Applied Mathematics program under award number ERKJE45; and by the Laboratory Directed Research and Development program at the Oak Ridge National Laboratory, which is operated by UT-Battelle, LLC., for the U.S.~Department of Energy under Contract DE-AC05-00OR22725.
}
\and
Junping Wang\thanks{Division of Mathematical Sciences, National
Science Foundation, Arlington, VA 22230 (jwang@\break nsf.gov). The
research of Wang was supported by the NSF IR/D program, while
working at the Foundation. However, any opinion, finding, and
conclusions or recommendations expressed in this material are those
of the author and do not necessarily reflect the views of the
National Science Foundation.}
\and Xiu Ye\thanks{Department of
Mathematics, University of Arkansas at Little Rock, Little Rock, AR
72204 (xxye@ualr.edu). This research was supported in part by
National Science Foundation Grant DMS-1115097}}
\begin{document}
\maketitle

\begin{abstract}
A discrete divergence free weak Galerkin finite element method is developed for the Stokes equations based on a weak Galerkin (WG)  method introduced in \cite{wy-stokes}. Discrete divergence free bases are constructed explicitly  for the lowest order weak Galerkin elements in  two and three dimensional spaces. These basis functions can be derived on general meshes of arbitrary shape of polygons and polyhedrons.
With the divergence free basis derived, the discrete divergence free WG scheme can eliminate pressure variable from the system and reduces a saddle point problem to a symmetric and positive definite system with many fewer unknowns. Numerical results are presented to demonstrate the robustness and accuracy of this discrete divergence free WG method.
\end{abstract}

\begin{keywords}
Weak Galerkin, finite element methods, the Stokes equations,
divergence free.
\end{keywords}

\begin{AMS}
Primary, 65N15, 65N30, 76D07; Secondary, 35B45, 35J50
\end{AMS}
\pagestyle{myheadings}

\section{Introduction}
The Stokes problem
seeks unknown functions $\bu$ and $p$ satisfying
\begin{eqnarray}
-\nabla\cdot A\nabla\bu+\nabla p&=& \bbf\quad
\mbox{in}\;\Omega,\label{moment}\\
\nabla\cdot\bu &=&0\quad \mbox{in}\;\Omega,\label{cont}\\
\bu&=&0\quad \mbox{on}\; \partial\Omega,\label{bc}
\end{eqnarray}
where $\Omega$ is a polygonal domain in
$\mathbb{R}^d$ with $d=2,3$ and $A$ is a symmetric and positive definite $d\times d$ matrix-valued
function in $\Omega$.
For the nonhomogeneous boundary condition
\[
\bu=\bg\quad \mbox{on}\; \partial\Omega,
\]
one can use the standard procedure by letting $\bu=\bu_0+\bu_g$. $\bu_g$ is a known function satisfying $\bu_g=\bg$ on $\partial\Omega$ and $\bu_0$ is zero at $\partial\Omega$ and satisfies (\ref{moment})-(\ref{cont}) with different right hand sides.

The weak form in the primary velocity-pressure formulation for the
Stokes problem (\ref{moment})--(\ref{bc}) seeks $\bu\in
[H_0^1(\Omega)]^d$ and $p\in L_0^2(\Omega)$ satisfying
\begin{eqnarray}
(A\nabla\bu,\nabla\bv)-(\nabla\cdot\bv, p)&=&({\bf f}, \bv),\quad \forall\bv\in [H_0^1(\Omega)]^d\label{w1}\\
(\nabla\cdot\bu, q)&=&0,\quad\quad\forall  q\in L_0^2(\Omega).\label{w2}
\end{eqnarray}

In the standard finite element methods for the Stokes and the
Navier-Stokes equations, both pressure and velocity are approximated
simultaneously. The primitive system is a large saddle point
problem. Numerical solvers for such indefinite systems are usually
less effective and robust than solvers for definite systems. On the
other hand, the divergence-free finite element method, discrete or
exact, computes numerical solution of  velocity by solving a
symmetric positive definite system in a divergence-free subspace. It
eliminates the pressure from the coupled equations and hence
significantly reduces the size of the system. The divergence-free
method is particularly attractive in the cases where the velocity is
the primary variable of interest, for example, the groundwater flow
calculation. The main tasks in the implementation of the
divergence-free method are to understand divergence-free subspaces,
weakly or exactly, and to construct bases for them.

Many finite element methods, continuous \cite{cr,gr,gun} and
discontinuous \cite{ckss,grw,kj,ls, wy-hdiv}, have been developed
and analyzed for the Stokes and the Navier-Stokes equations.
Divergence-free basis for different finite element methods have been
constructed \cite{gri1, gri2, gri3,gh1,gh2,wwy,yh1,yh2}.

A weak Galerkin finite element method  was introduced in
\cite{wy-stokes} for the Stokes equations in the primal
velocity-pressure formulation. This method is designed by using
discontinuous piecewise polynomials on finite element partitions
with arbitrary shape of polygons/polyhedra. Weak Galerkin methods
were  first introduced  in \cite{wy, wy-m} for second order elliptic
equations. In general, weak Galerkin finite element formulations for
partial differential equations can be derived naturally by replacing
usual derivatives by weakly-defined derivatives in the corresponding
variational forms, with an option of adding a stabilization term to
enforce a weak continuity of the approximating functions.  Therefore
the weak Galerkin method developed in \cite{wy-stokes} for the
Stokes equations  naturally has the form: find $\bu_h\in V_h$ and
$p_h\in W_h$  satisfying
\begin{eqnarray}
(A\nabla_w\bu_h,\nabla_w\bv)+s(\bu_h,\bv)-(\nabla_w\cdot\bv,p_h)&=&({\bf f},\bv),\label{w3}\\
(\nabla_w\cdot\bu_h,q)&=&0\label{w4}
\end{eqnarray}
for all the test functions $\bv\in V_h$ and $q\in W_h$  where $V_h$ and $W_h$ will be defined later. The stabilizer  $s(\bu_h,\bv)$ in (\ref{w3}) is parameter independent.

Let $D_h$ be a discrete divergence free subspace of $V_h$ such that $(\nabla_w\cdot\bv,q)=0$ for ant $q\in W_h$.
Then the discrete divergence free WG formulation  is to  find $\bu_h\in D_h$   satisfying
\begin{eqnarray}
(A\nabla_w\bu_h,\nabla_w\bv)+s(\bu_h,\bv)&=&({\bf f},\bv),\quad\forall\bv\in D_h.\label{dfw}
\end{eqnarray}
System (\ref{dfw}) is symmetric and positive definite with many fewer unknowns. The main purpose of this paper is to construct  bases for $D_h$ in two and three dimensional spaces. A unique feature of these divergence free basis functions is that they can be obtained on general meshes such as hybrid meshes or meshes with hanging nodes. Numerical examples in two dimensional space are provided to confirm the theory. Although the Stokes equations is considered, the divergence free basis can be use for solving the Navier-Stokes equations.

\section{A Weak Galerkin Finite Element Method}\label{Section:wg-fem}

In this section, we will review the WG method for the Stokes
equations introduced in \cite{wy-stokes} with $k=1$.

Let ${\cal T}_h$ be a partition of the domain $\Omega$ consisting
mix of polygons satisfying a set of conditions specified in
\cite{wy-m}. In addition, we assume that all the elements $T\in\T_h$
are convex. Denote by ${\cal F}_h$ the set of all edges in 2D or
faces in 3D in ${\cal T}_h$, and let ${\cal F}_h^0={\cal
F}_h\backslash\partial\Omega$ be the set of all interior edges or
faces.

We define a weak Galerkin finite element space for
the velocity as follows
\[
V_h=\left\{ \bv=\{\bv_0, \bv_b\}:\ \{\bv_0, \bv_b\}|_{T}\in
[P_{1}(T)]^d\times [P_{0}(e)]^d,\ e\subset\pT,\ \bv_b=0\ \mbox{on}\
\partial\Omega \right\}.
\]
We would like to emphasize that there is only a single value $\bv_b$
defined on each edge in 2D and face in 3D. For the pressure variable, we have
the following finite element space
\[
W_h=\left\{q:\ q\in L_0^2(\Omega), \ q|_T\in P_{0}(T)\right\}.
\]

For a given $\bv\in V_h$, a  weak gradient and a  weak divergence  are defined locally on each $T\in \T_h$ as follows.

\smallskip

\begin{defi}
A  weak gradient, denoted by
$\nabla_{w}$, is defined as the unique polynomial
$(\nabla_{w}\bv) \in [P_{0}(T)]^{d\times d}$ for $\bv\in V_h$ satisfying the
following equation,
\begin{equation}\label{d-g}
(\nabla_{w}\bv, q)_T = -(\bv_0,\nabla\cdot q)_T+ \langle \bv_b,
q\cdot\bn\rangle_{\partial T},\qquad \forall q\in [P_{0}(T)]^{d\times d},
\end{equation}
\end{defi}

\begin{defi}
A  weak divergence, denoted by
($\nabla_{w}\cdot$), is defined as the unique polynomial
$(\nabla_{w}\cdot\bv) \in P_{0}(T)$ for $\bv\in V_h$ that satisfies the following
equation
\begin{equation}\label{d-d}
(\nabla_{w}\cdot\bv, \varphi)_T = -(\bv_0, \nabla\varphi)_T + \langle
\bv_b\cdot\bn, \varphi\rangle_{\partial T},\qquad
\forall \varphi\in P_{0}(T).
\end{equation}
\end{defi}

Denote by $Q_{0}$ the $L^2$ projection operator from $[L^2(T)]^d$
onto $[P_1(T)]^d$ and denote by
$Q_{b}$ the $L^2$ projection from $[L^2(e)]^d$ onto $[P_{0}(e)]^d$.
Let $(\bu;p)$ be the solution of (\ref{moment})-(\ref{bc}).
Define $Q_h\bu=\{Q_0\bu, Q_b\bu\}\in V_h$.
let $\bbQ_h$ be the  local $L^2$
projections onto $P_{0}(T)$.

\smallskip

We  introduce three bilinear forms as follows
\begin{eqnarray*}
s(\bv,\;\bw) &= &\sum_{T\in {\cal T}_h}h_T^{-1}\l Q_b\bv_0-\bv_b,\;\;Q_b\bw_0-\bw_b\r_\pT,\\
a(\bv,\ \bw)&=&\sum_{T\in\T_h}(A\nabla_w\bv,\ \nabla_w\bw)_T+s(\bv,\bw),\\
b(\bv,\ q)&=&\sum_{T\in\T_h}(\nabla_w\cdot\bv,\ q)_T.
\end{eqnarray*}

\begin{algorithm}
A numerical approximation for (\ref{moment})-(\ref{bc})
can be obtained by seeking $\bu_h=\{\bu_0,\bu_b\}\in V_h$ and
$p_h\in W_h$ such that
\begin{eqnarray}
a(\bu_h,\ \bv)-b(\bv,\;p_h)&=&(f,\;\bv_0),\quad\forall\bv\in V_h\label{wg1}\\
b(\bu_h,\;q)&=&0,\quad\forall q\in W_h.\label{wg2}
\end{eqnarray}
\end{algorithm}

Define
\begin{equation}\label{3barnorm}
\3bar\bv\3bar^2=a(\bv,\bv).
\end{equation}

The following optimal error estimates have been derived in \cite{wy-stokes}.
\begin{theorem}\label{h1-bd}
Let $(\bu; p)\in  [H_0^1(\Omega)\cap H^{2}(\Omega)]^d\times
(L_0^2(\Omega)\cap H^{1}(\Omega))$  and $(\bu_h;p_h)\in
V_h\times W_h$ be the solution of (\ref{moment})-(\ref{bc}) and
(\ref{wg1})-(\ref{wg2}), respectively. Then, the following error
estimates hold true
\begin{eqnarray}
\3bar Q_h\bu-\bu_h\3bar+\|\bbQ_hp-p_h\|&\le& Ch(\|\bu\|_{2}+\|p\|_{1}),\label{err1}\\
\|Q_0\bu-\bu_0\|&\le& Ch^{2}(\|\bu\|_{2}+\|p\|_{1}).\label{l2-error}
\end{eqnarray}
\end{theorem}

Define a discrete  divergence free subspace $D_h$ of $V_h$ by
\begin{equation}\label{D_h}
D_h=\{\bv\in V_h;\;b(\bv,q) =0,\quad\forall q\in W_h\}.
\end{equation}

By taking the test functions from $D_h$,
the weak Galerkin formulation (\ref{wg1})-(\ref{wg2}) is equivalent to the following  divergence-free weak Galerkin finite element scheme.

\smallskip

\begin{algorithm}\label{algorithm2}
A discrete divergence free WG approximation for (\ref{moment})-(\ref{bc}) is to find
$\bu_h=\{\bu_0,\bu_b\}\in D_h$ such that
\begin{eqnarray}
a(\bu_h,\ \bv)&=&(f,\;\bv_0),\quad \forall \bv=\{\bv_0,\bv_b\}\in D_h.\label{df-wg}
\end{eqnarray}
\end{algorithm}

System (\ref{df-wg}) is symmetric and positive definite with many
fewer unknowns. It can be solved effectively by many existing
solvers.

The main task of this paper is to construct basis for $D_h$. In the
next two sections, discrete divergence free bases will be
constructed explicitly for two and three dimensional spaces.

\section{Construction of Discrete Divergence Free Basis for Two Dimensional Space}

For a given partition $\T_h$,  let  ${\cal V}_h^0$ be the set of all
interior vertices. Let $N_F=card (\F_h^0)$, $N_V=card ({\cal
V}_h^0)$ and $N_K=card (\T_h)$. It is known based on the Euler
formula that for a partition consisting of convex polygons, then
\begin{equation}\label{key}
N_F+1=N_V+N_K.
\end{equation}
For a mesh $\T_h$ with  hanging nodes, the relation in (\ref{key}) is still true if we treat the hanging nodes as vertices.

First we need to derive a basis  for $V_h$. For each $T\in\T_h$ and
any $\bv=\{\bv_0,\bv_b\}\in V_h$, $\bv_0$ is a vector function with
two components and each component is a linear function. Therefore
there are six linearly independent linear functions $\Phi_{j+1},
\Phi_{j+2},\cdots,\Phi_{j+6}$ in $V_h$ such that they are nonzero
only at the interior of element $T$. For each $e_i\in\F_h^0$,
$\bv_b$ is a constant vector function. Thus, there are two linearly
independent constant  functions $\Psi_{i,1}$ and $\Psi_{i,2}$ in
$V_h$ which take nonzero value only on $e_i$. Then it is easy to see
that
\begin{equation}\label{basis}
V_h={\rm span}\{\Phi_1,\cdots, \Phi_{6N_K},\Psi_{1,1},\Psi_{1,2},\cdots,\Psi_{N_F,1},\Psi_{N_F,2}\}.
\end{equation}
For a given function $\bv=\{\bv_0,\bv_b\}\in V_h$, it is easy to see that $\bv_0$ can be spanned by the basis functions $\Phi_i$ and $\bv_b$ by the  basis functions $\Psi_{j,k}$.

Next, we will find the dimension of $D_h$. Since the dimension for pressure space $W_h$ is $N_K-1$, it follows from (\ref{key}) that
\begin{equation}\label{bbb}
\rm{dim} (D_h)=\rm{dim} (V_h)-\rm{dim}(W_h)=6N_K+2N_F-N_K+1=6N_K+N_F+N_V.
\end{equation}

\begin{lemma}
The basis functions $\Phi_1,\cdots, \Phi_{6N_K}$ of $V_h$ in (\ref{basis}) are in $D_h$ and linearly independent.
\end{lemma}

\smallskip

\begin{proof}
Let $\Phi_i=\{\Phi_{i,0},\Phi_{i,b}\}$. The definition of $\Phi_i$ implies $\Phi_{i,b}=0$.
 For any $q\in W_h$, it follows from (\ref{d-d}), $\nabla q=0$ and $\Phi_{i,b}=0$,
\begin{eqnarray*}
b(\Phi_i,q)&=&\sum_{T\in\T_h} (\nabla_w\cdot\Phi_i,q)_T\\
&=&\sum_{T\in\T_h}(-(\Phi_{i,0},\nabla q)_T+\l \Phi_{i,b}\cdot\bn, q\r_\pT)\\
&=&0,
\end{eqnarray*}
which proves the lemma since the linear independence of  $\Phi_1,\cdots, \Phi_{6N_K}$ is obvious.
\end{proof}

\smallskip

For any $e_i\in\F_h^0$,let $\psi_{i,1}$ and $\psi_{i,2}$ be two basis functions of $V_h$ associated with $e_i$. Let $\bn_{e_i}$ and $\bt_{e_i}$ be a normal vector and a tangential vector to $e_i$ respectively. Define  $\Upsilon_i=C_1\Psi_{i,1}+C_2\Psi_{i,2}$ such that $\Upsilon_i|_{e_i}=\bt_{e_i}$. Obviously $\Upsilon_i\in V_h$ is only nonzero on $e_i$.

\smallskip

\begin{lemma}
Functions  $\Upsilon_1,\cdots,\Upsilon_{N_F}\in V_h$ are in $D_h$ and linearly independent.
\end{lemma}

\smallskip

\begin{proof}
Let $\Upsilon_i=\{\Upsilon_{i,0},\Upsilon_{i,b}\}$. For any $q\in W_h$, it follows from (\ref{d-d}) and $\nabla q=0$,
\begin{eqnarray*}
b(\Upsilon_i,q)&=&\sum_{T\in\T_h} (\nabla_w\cdot\Upsilon_i,q)_T\\
&=&\sum_{T\in\T_h}(-(\Upsilon_{i,0},\nabla q)_T+\l \Upsilon_{i,b}\cdot\bn, q\r_\pT)\\
&=&\sum_{T\in\T_h}\l \Upsilon_{i,b}\cdot\bn, q\r_\pT\\
&=&0,
\end{eqnarray*}
where we use the fact $\bt_{e_i}\cdot\bn=0$.
Since $\Upsilon_i$ is only nonzero on $e_i$, $\Upsilon_1,\cdots,\Upsilon_{N_F}$ are linearly independent. We completed the proof.
\end{proof}

\smallskip

For a given interior vertex $P_i\in {\cal V}_h^0$, assume that there are $r$ elements having $P_i$ as a vertex which form a hull ${\cal H}_{P_i}$ as shown in Figure \ref{fig1}. Then there are $r$ interior edges $e_j$ ($j=1,\cdots, r$) associated with ${\cal H}_{P_i}$. Let $\bn_{{e_j}}$ be a normal vector on $e_j$ such that normal vectors $\bn_{e_j}$ $j=1,\cdots,r$ are counterclockwise around vertex $P_i$ as shown in Figure \ref{fig1}.  For each $e_j$, let $\Psi_{j,1}$ and $\Psi_{j,2}$ be the two basis functions of $V_h$ which is only nonzero on $e_j$. Define $\Theta_j=C_1\Psi_{j,1}+C_2\Psi_{j,2}\in V_h$ such that $\Theta_j|_{e_j}=\bn_{e_j}$. Define $\Lambda_i=\sum_{j=1}^r \frac{1}{|e_j|}\Theta_j$.

\begin{figure}[!tb]
\centering
\includegraphics[width=0.4\textwidth]{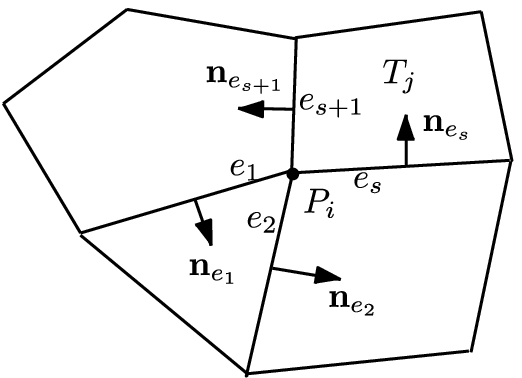}\qquad \qquad
\includegraphics[width=0.4\textwidth]{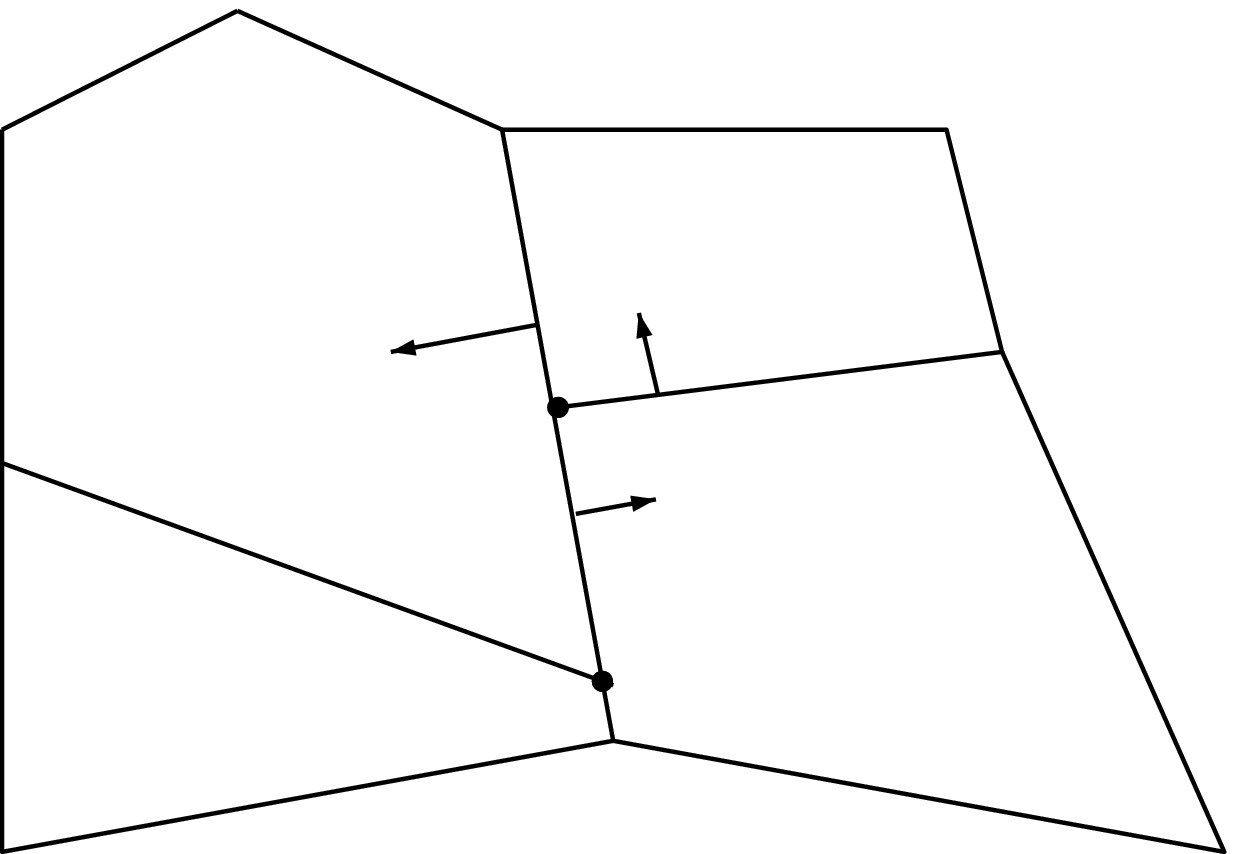}\\
  (a)\quad\quad\quad\quad\quad\quad\quad\quad\quad\quad\quad\quad\quad\quad\quad   (b)\\
\caption{(a) A 2D hull ${\cal H}_{P_i}$; (b) A hull with hanging node.}
\label{fig1}
\end{figure}

\smallskip

\begin{lemma}\label{lemma3}
Functions $\Lambda_1,\cdots,\Lambda_{N_V}\in V_h$ are in $D_h$ and linearly independent.
\end{lemma}

\smallskip

\begin{proof}
 Suppose that there exist constants  $c_1,\cdots, c_{N_V}$ such that $\sum_{i=1}^{N_V}c_i\Lambda_i=0$. Let $\Lambda_l$ be associated with a hull ${\cal H}_{P_l}$ such  that there exists $e_m$ as one of the interior edges of ${\cal H}_{P_l}$  and edge $e_m$ has  a boundary node as one of its end points. Since $\Lambda_i=0$ on $e_m$ for $i\neq l$,  we have
\[
0=\sum_{i=1}^{N_V}\int_{e_m}c_i\Lambda_i=\int_{e_m}c_l\Lambda_l=\int_{e_m}c_l\Theta_m=c_l\bn_{e_m},
\]
which implies $c_l=0$. By this way, we can prove that all $c_i=0$ and $\Lambda_1,\cdots,\Lambda_{N_V}$ are linearly independent.
Next, we will show that $b(\Lambda_i,q)=0$ for all $q\in W_h$. Let $q_j\in W_h$ such that $q_j=1$ on $T_j\in\T_h$ and $q_j=0$ otherwise. So we only need to show that
\begin{equation}\label{qi}
b(\Lambda_i,q_j)=0,\quad\forall j=1,\cdots, N_K.
\end{equation}
Let $\Lambda_i$ and $q_j$ be associated with hull ${\cal H}_{P_i}$ and element $T_j$ respectively.
If $T_j\in\T_h$ is not in ${\cal H}_{P_i}$, we easily have $b(\Lambda_i,q_j)=0$. If $T_j\in {\cal H}_{P_i}$, let $e_s$ and $e_{s+1}$ be its two edges in ${\cal H}_{P_i}$ shown in Figure \ref{fig1},
\begin{eqnarray*}
b(\Lambda_i,q_j)&=&\sum_{T\in\T_h} (\nabla_w\cdot\Lambda_i,q_j)_T\\
&=& (\nabla_w\cdot\Lambda_i,q_j)_{T_j}\\
&=&-(\Lambda_{i,0},\nabla q)_{T_j}+\l \Lambda_{i,b}\cdot\bn, q\r_{\pT_j}\\
&=&\int_{e_s}\frac{1}{|e_s|}\bn_{e_s}\cdot\bn+\int_{e_{s+1}}\frac{1}{|e_{s+1}|}\bn_{e_{s+1}}\cdot\bn\\
&=&0.
\end{eqnarray*}
We proved $\Lambda_i\in D_h$ for $i=1,\cdots, N_V$.
\end{proof}

\smallskip

\begin{theorem}\label{thm1}
Let $D_h$ be defined in (\ref{D_h}). Then for two dimensional space, $D_h$ is spanned by the following basis functions,
\begin{equation}\label{main}
D_h={\rm Span}\{\Phi_1,\cdots, \Phi_{6N_K}, \Upsilon_1,\cdots, \Upsilon_{N_F},\Lambda_1,\cdots,\Lambda_{N_V}\}.
\end{equation}
\end{theorem}

\begin{proof}
The number of the functions in the right hand side of (\ref{main}) is $6N_K+N_F+N_V$ which is equal to ${\rm dim}(D_h)$ due to (\ref{bbb}).
Next, we prove that
$$\Phi_1,\cdots, \Phi_{6N_K}, \Upsilon_1,\cdots, \Upsilon_{N_F},\Lambda_1,\cdots,\Lambda_{N_V}$$
are linearly independent.
Since $\Phi_i$ take zero value on all $f\in\F_h$, $\Phi_i$ will be linearly independent to all $\Upsilon_l$ and $\Lambda_m$.
Suppose
\begin{equation}\label{t1}
C_1\Upsilon_1+\cdots+C_{N_F}\Upsilon_{N_F}+C_{N_F+1}\Lambda_1+\cdots+C_{N_F+N_V}\Lambda_{N_V}=0.
\end{equation}
Multiplying (\ref{t1}) by $\Upsilon_i$ and integrating over $e_i$, we have
\[
C_i|e_i|=0,
\]
where we use the fact $\bt_{e_i}\cdot\bn_e=0$. Thus we can obtain $C_i=0$ for $i=1,\cdots,N_F$. By Lemma \ref{lemma3}, we can prove $C_i=0$ for $i=N_F+1,\cdots,N_F+N_V$. The proof of the lemma is completed.
\end{proof}

\section{Construction of Discrete Divergence Free Basis for  Three Dimensional Space}

Let $\T_h$ be a partition of $\Omega\subset \mathbb{R}^3$ consisting polyhedrons without hanging nodes.  Recall $N_F=card (\F_h^0)$, $N_V=card ({\cal V}_h^0)$ and $N_K=card (\T_h)$.  Denote by $\E_h$ all the edges in $\T_h$ and let $\E_h^0={\cal E}_h\backslash\partial\Omega$. Let $N_E=card (\E_h^0)$.

It is known based on the Euler formula that for a  partition consisting of convex polyhedrons, then
\begin{equation}\label{key3}
N_V+N_F+1=N_E+N_K.
\end{equation}

For each $T\in\T_h$ and any $\bv=\{\bv_0,\bv_b\}\in V_h$, $\bv_0$ is a vector function with three components and each component is a linear function. Therefore
there are twelve linearly independent linear functions $\Phi_{j+1}, \Phi_{j+2},\cdots,\Phi_{j+12}$ in $V_h$ such that they are nonzero only at the interior of element $T$. For each face $f_i\in\F_h^0$, $\bv_b$ is a constant vector function with three component. Thus there are three linearly independent constant vector functions $\Psi_{i,1}$, $\Psi_{i,2}$ and $\Psi_{i,3}$ in $V_h$ which take nonzero value only on the face $f_i$. Then it is easy to see that
\begin{equation}\label{basis3}
V_h={\rm span}\{\Phi_1,\cdots, \Phi_{12N_K},\Psi_{1,1},\Psi_{1,2},\Psi_{1,3}\cdots,\Psi_{N_F,1},\Psi_{N_F,2},\Psi_{N_F,3}\}.
\end{equation}

Since the dimension for pressure space $W_h$ is $N_K-1$, (\ref{key3}) implies
\begin{eqnarray}
\rm{dim} (D_h)&=&\rm{dim} (V_h)-\rm{dim}(W_h)\label{bbb3}\\
&=&12N_K+3N_F-N_K+1=12N_K+2N_F+N_E-N_V.\nonumber
\end{eqnarray}

\begin{lemma}
The  functions $\Phi_1,\cdots, \Phi_{12N_K}$  in (\ref{basis3}) are in $D_h$ and linearly independent.
\end{lemma}

\smallskip

\begin{proof}
Let $\Phi_i=\{\Phi_{i,0},\Phi_{i,b}\}$. The definition of $\Phi_i$ implies $\Phi_{i,b}=0$.
 For any $q\in W_h$, it follows from (\ref{d-d}), $\nabla q=0$ and $\Phi_{i,b}=0$ that for any $q\in W_h$,
\begin{eqnarray*}
b(\Phi_i,q)&=&\sum_{T\in\T_h} (\nabla_w\cdot\Phi_i,q)_T\\
&=&\sum_{T\in\T_h}(-(\Phi_{i,0},\nabla q)_T+\l \Phi_{i,b}\cdot\bn, q\r_\pT)\\
&=&0,
\end{eqnarray*}
which finished the proof of the lemma since the linear independence of  $\Phi_1,\cdots, \Phi_{12N_K}$ is obvious.
\end{proof}

\smallskip

For any face $f_i\in\F_h^0$, let $\bn$  be a unit normal vector of $f_i$ and let $\bt_{1}$ and $\bt_{2}$ be two linearly independent unit tangential vectors to the face $f_i$. Define  $\Upsilon_{i,1}=C_{1,1}\Psi_{i,1}+C_{1,2}\Psi_{i,2}+C_{1,3}\Psi_{i,3}$ and $\Upsilon_{i,2}=C_{2,1}\Psi_{i,1}+C_{2,2}\Psi_{i,2}+C_{2,3}\Psi_{i,3}$ such that $\Upsilon_{i,1}|_{f_i}=\bt_{1}$ and $\Upsilon_{i,1}|_{f_i}=\bt_{2}$ respectively. Obviously $\Upsilon_{i,1}$ and $\Upsilon_{i,2}$ are in $V_h$ and only nonzero on $f_i$.

\smallskip

\begin{lemma}
Functions  $\Upsilon_{1,1},\Upsilon_{1,2}\cdots,\Upsilon_{N_F,1}, \Upsilon_{N_F,2}\in V_h$ are in $D_h$ and linearly independent.
\end{lemma}

\smallskip

\begin{proof}
Let $\Upsilon_{i,j}=\{\Upsilon_{i,j,0},\Upsilon_{i,j,b}\}$ with $j=1,2$. For any $q\in W_h$, it follows from (\ref{d-d}) and $\nabla q=0$ for $j=1,2$,
\begin{eqnarray*}
b(\Upsilon_{i,j},q)&=&\sum_{T\in\T_h} (\nabla_w\cdot\Upsilon_{i,j},q)_T\\
&=&\sum_{T\in\T_h}(-(\Upsilon_{i,j,0},\nabla q)_T+\l \Upsilon_{i,j,b}\cdot\bn, q\r_\pT)\\
&=&\sum_{T\in\T_h}\l \Upsilon_{i,j,b}\cdot\bn, q\r_\pT\\
&=&0,
\end{eqnarray*}
where we use the fact $\bt_{j}\cdot\bn=0$ with $j=1,2$.
Since $\Upsilon_{i,j}$ is only nonzero on $f_i$, $\Upsilon_{1,1},\Upsilon_{1,2}\cdots,\Upsilon_{N_F,1}, \Upsilon_{N_F,2}$ are linearly independent. We completed the proof.
\end{proof}

\begin{figure}
\begin{center}
\includegraphics[width=6cm]{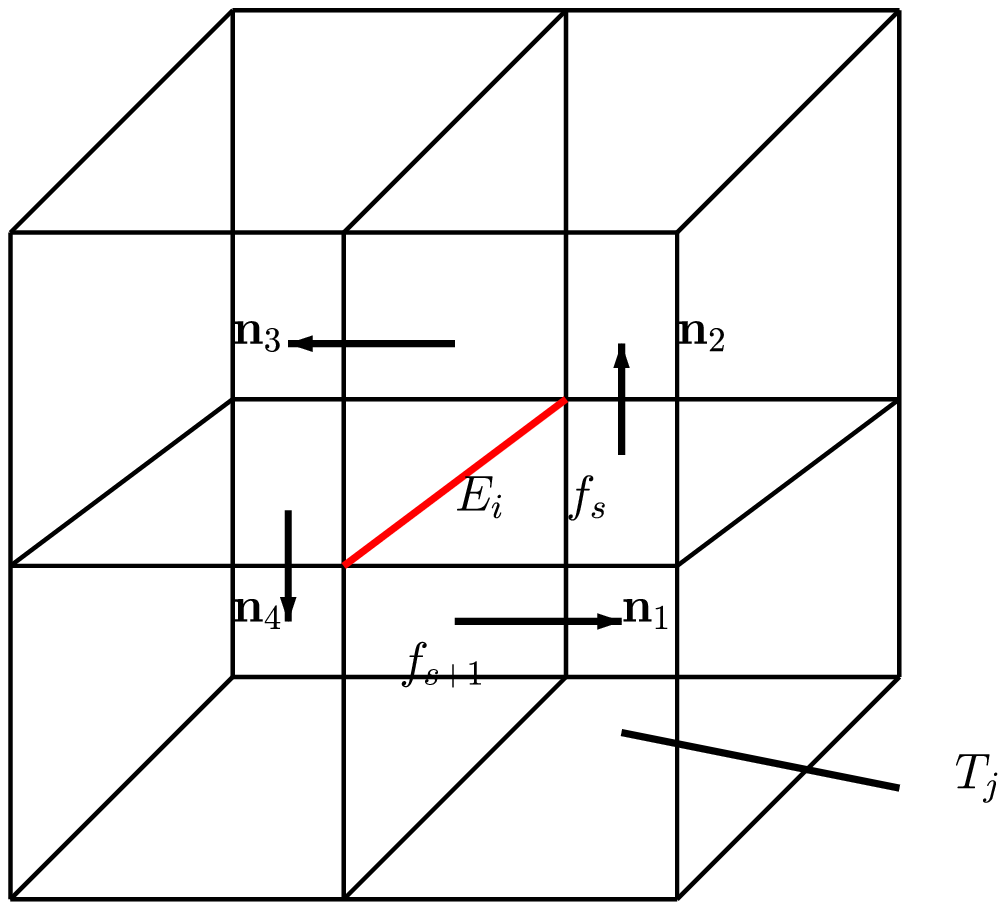}
\caption{A 3D Hull ${\cal S}_{E_i}$ .}  \label{fig2}
\end{center}
\end{figure}

\smallskip

For a given interior edge $E_i\in {\cal E}_h$,  assume there are $r$ elements having $E_i$ as one of their edges which form a solid denoted by ${\cal S}_{E_i}$ shown in Figure \ref{fig2}. Then there are $r$ interior faces $f_j$ ($j=1,\cdots, r$) in ${\cal S}_{E_i}$. Let $\bn_j$ be a unit normal vector on the face $f_j$ such that normal vectors $\bn_{j}$ $j=1,\cdots,r$ form oriented loop around the interior edge $E_i$ shown in Figure \ref{fig2}.  For each $f_j$, let $\Psi_{j,1}$, $\Psi_{j,2}$ and $\Psi_{j,3}$ be the three basis functions of $V_h$ which are only nonzero on $f_j$. Define $\Theta_j=C_1\Psi_{j,1}+C_2\Psi_{j,2}+C_3\Psi_{j,3}\in V_h$ such that $\Theta_j|_{f_j}=\bn_{j}$. Define $\Lambda_i=\sum_{j=1}^r \frac{1}{|f_j|}\Theta_j$.

\smallskip

\begin{lemma}
Functions $\Lambda_1,\cdots,\Lambda_{N_E}\in V_h$ are in $D_h$ .
\end{lemma}

\smallskip

\begin{proof}
Let $\Lambda_i$ and $q_j$ be associated with hull ${\cal S}_{E_i}$ and element $T_j$ respectively.
If $T_j\in\T_h$ is not in ${\cal S}_{E_i}$, we easily have $b(\Lambda_i,q_j)=0$. If $T_j\in {\cal S}_{E_i}$, let faces $f_s$ and $f_{s+1}$ be its two faces in ${\cal S}_{E_i}$ shown in Figure \ref{fig2},
\begin{eqnarray*}
b(\Lambda_i,q_j)&=&\sum_{T\in\T_h} (\nabla_w\cdot\Lambda_i,q_j)_T\\
&=& (\nabla_w\cdot\Lambda_i,q_j)_{T_j}\\
&=&-(\Lambda_{i,0},\nabla q)_{T_j}+\l \Lambda_{i,b}\cdot\bn, q\r_{\pT_j}\\
&=&\int_{f_s}\frac{1}{|f_s|}\bn_{f_s}\cdot\bn+\int_{f_{s+1}}\frac{1}{|f_{s+1}|}\bn_{f_{s+1}}\cdot\bn\\
&=&0
\end{eqnarray*}
We proved $\Lambda_i\in D_h$ for $i=1,\cdots, N_E$.
\end{proof}

\smallskip

Unfortunately, $\Lambda_1,\cdots,\Lambda_{N_E}$ are linearly dependent. Let $P_i$ be an interior vertex in $\T_h$ and ${\cal G}_{P_i}$ be a hull formed by the elements $T\in\T_h$ sharing $P_i$. Let $e_j\in\E_h^0,\;j=1,\cdots,t$ with $P_i$ as one of its end point and $\Lambda_j,\; j=1,\cdots, t$ be the discrete divergence free functions associated with $e_j$. With appropriate choosing $\bn_j$ in defining $\Lambda_j$, one can prove that $\sum_{j=1}^t\Lambda_j=0$. However, if we eliminate one function from $\{\Lambda_1,\cdots,\Lambda_t\}$ randomly, say $\Lambda_1$, we will prove that $\{\Lambda_2,\cdots,\Lambda_t\}$ are linearly independent in the following lemma.

\begin{lemma}\label{lemma9}
Functions $\Lambda_2,\cdots,\Lambda_t$ are linearly independent.
\end{lemma}

\smallskip

\begin{proof}
Let $f\in\F_h^0$ be an interior face in ${\cal G}_{P_i}$ with $e_1$ and $e_2$ as its two edges in ${\cal G}_{P_i}$. The definition of $\Lambda_j$ implies that only $\Lambda_1$ and $\Lambda_2$ are nonzero on $f$. Suppose that there exist constants  $c_2,\cdots, c_{t}$ such that $\sum_{i=2}^{t}c_i\Lambda_i=0$.  Then we have
\[
0=\sum_{i=2}^{t}\int_{f}c_i\Lambda_i=\int_{f}c_2\Lambda_2=\int_{f}c_2\Theta_f=c_2\bn_{f},
\]
which implies $c_2=0$. By this way, we can prove that all $c_i=0$ and $\Lambda_2,\cdots,\Lambda_{t}$ are linearly independent.
\end{proof}

We start with $\{\Lambda_1,\cdots,\Lambda_{N_E}\}$ and eliminate one function for each ${\cal G}_{P_i}$ for $i=1,\cdots, N_V$. With renumbering the functions, we end up with $N_E-N_V$ discrete divergence free functions: $\{\Lambda_1,\cdots,\Lambda_{N_E-N_V}\}$.

\smallskip

\begin{lemma}\label{lemma1}
Functions $\{\Lambda_1,\cdots,\Lambda_{N_E-N_V}\}$ are linearly independent.
\end{lemma}

\begin{proof}
The proof of the lemma is similar to the proofs of Lemma \ref{lemma3} and Lemma \ref{lemma9}.
\end{proof}

\smallskip

\begin{theorem}
Let $D_h$ be defined in (\ref{D_h}). Then for  three dimensional space, $D_h$ is spanned by the following basis functions,
\begin{equation}\label{main3}
D_h={\rm Span}\{\Phi_1,\cdots, \Phi_{12N_K}, \Upsilon_{1,1},\Upsilon_{1,2},\cdots, \Upsilon_{N_F,1},\Upsilon_{N_F,2},\Lambda_1,\cdots,\Lambda_{N_E-N_V}\}.
\end{equation}
\end{theorem}
\begin{proof}
The number of the functions in the right hand side of (\ref{main3}) is $12N_K+2N_F+N_E-N_V$ which is equal to ${\rm dim}(D_h)$ due to (\ref{bbb3}).
Similar to the proof of Theorem \ref{thm1}, we can prove that  $\{\Phi_1,\cdots, \Phi_{6N_K}, \Upsilon_{1,1},\Upsilon_{1,2},\cdots, \Upsilon_{N_F,1},\Upsilon_{N_F,2},\Lambda_1,\cdots,\Lambda_{N_E-N_V}\}$ are linear independent.
 \end{proof}

\section{Numerical Experiments}
In this section, we shall report several results of numerical examples for  two dimensional Stokes equations. The divergence-free finite element scheme introduced in Algorithm \ref{algorithm2} is used. The main purpose if to numerically validate the accuracy and efficiency of the WG scheme.

Let $\bv_h\in D_h$ and $q_h\in W_h$, the error for the WG-FEM solution is measured in three norms defined
as follows:
\begin{eqnarray*}
\3bar \bv_h \3bar^2:&=&\sum_{T\in\mathcal{T}_h}\bigg(\int_T|\nabla_w \bv_h|^2dT+
+h_T^{-1}\int_{\partial T}(\bv_0-\bv_b)^2ds\bigg),\qquad(\mbox{A discrete $H^1$-norm}),\notag\\
\|\bv_0\|^2:&=&\sum_{T\in\mathcal{T}_h}\int_T|\bv_0|^2dx,\qquad\qquad\qquad\qquad  (\mbox{Element-based $L^2$-norm}).
\end{eqnarray*}

\subsection{Test case 1}
\begin{figure}[!htb]
\centering
\begin{tabular}{c}
  \resizebox{2.2in}{2.2in}{\includegraphics{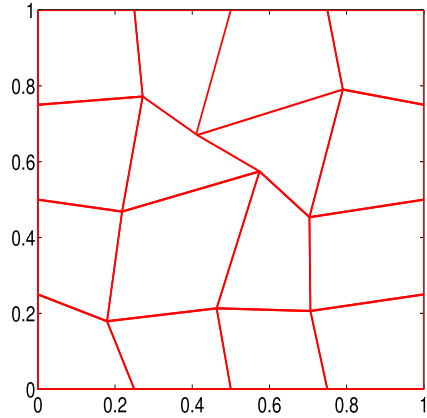}}
\end{tabular}
\caption{Example 1: Level 1 of mixed polygonal mesh.}\label{fig:ex1_1}
\end{figure}

\begin{table}[h]
\caption{Test Case 1: Numerical error and convergence rates for the
Stokes equation with homogeneous boundary
conditions on the uniform rectangular meshes.}\label{ex1_rect} \center
\begin{tabular}{|c|cc|cc|}
\hline
$h$ & $\3bar \bu_h-Q_h \bu\3bar$&order & $\|\bu_0-Q_0 \bu\|$ &order \\
\hline\hline
   1/4     &8.1050e-01 &            &2.9957e-01  &        \\ \hline
   1/8     &6.9698e-01 &2.1769e-01  &9.9634e-02  &1.5882  \\ \hline
   1/16    &4.4578e-01 &6.4479e-01  &3.1031e-02  &1.6829  \\ \hline
   1/32    &2.4452e-01 &8.6638e-01  &8.5507e-03  &1.8596  \\ \hline
   1/64    &1.2620e-01 &9.5424e-01  &2.2131e-03  &1.9500  \\ \hline
   1/128   &6.3751e-02 &9.8519e-01  &5.5968e-04  &1.9834  \\ \hline
\end{tabular}
\end{table}

\begin{table}[h]
\caption{Test Case 1: Numerical error and convergence rates for the
Stokes equation with homogeneous boundary
conditions on the mixed polygonal meshes.}\label{ex1_derect} \center
\begin{tabular}{|c|cc|cc|}
\hline
$h$ & $\3bar \bu_h-Q_h \bu\3bar$&order & $\|\bu_0-Q_0 \bu\|$ &order \\
\hline\hline
   4.1016e-01   &8.1917e-01 &            &3.0927e-01 &      \\ \hline
   2.0508e-01   &7.0386e-01 &2.1887e-01  &1.0421e-01 &1.5694\\ \hline
   1.0254e-01   &4.6002e-01 &6.1359e-01  &3.3478e-02 &1.6382\\ \hline
   5.1270e-02   &2.5560e-01 &8.4781e-01  &9.4392e-03 &1.8265\\ \hline
   2.5635e-02   &1.3230e-01 &9.5007e-01  &2.4560e-03 &1.9424\\ \hline
   1.2818e-02   &6.6890e-02 &9.8401e-01  &6.2217e-04 &1.9810\\ \hline
\end{tabular}
\end{table}

The domain is set as $\Omega=(0,1)\times(0,1)$. Let the exact solution $\bu$ and $p$ as follows,
\begin{eqnarray*}
\bu=\begin{pmatrix}
10x^2y(x-1)^2(2y-1)(y-1)\\
-10xy^2(2x-1)(x-1)(y-1)^2
\end{pmatrix}
\mbox{ and }p=10(2x-1)(2y-1).
\end{eqnarray*}
It is easy to check that homogeneous Dirichlet boundary condition is satisfied for this testing. The right hand side function $f$ is given to match the exact solutions.

The first test shall be performed on the uniform rectangular meshes and the mixed polygonal meshes. The uniform rectangular meshes are generated by partition the domain $\Omega$ into $n\times n$ sub-rectangles. The mesh size is denoted by $h=1/n.$ Moreover, the WG divergence free algorithm is also test on the mixed polygonal type meshes. We start with the initial mesh shown as the Figure \ref{fig:ex1_1}, which contains the mixture of triangles and quadrilaterals. The next level of mesh is to refine the previous level of mesh by connecting the mid-point on each edge. The mesh size in this case is also denoted by $h$.

The error profile is reported in Table \ref{ex1_rect}-\ref{ex1_derect} for the rectangular meshes and mixed polygonal meshes, respectively. Both of the tables show the same convergence rate as the theoretical conclusion, which is $O(h)$ in the $H^1-$norm and $O(h^2)$ in the $L^2-$norm.

\subsection{Test case 2} The domain is given by $\Omega=(0,1)\times(0,1)$. Let the exact solutions $\bu$ and $p$ as follows,
\begin{eqnarray*}
\bu=\begin{pmatrix}
x(1-x)(1-2y)\\
-y(1-y)(1-2x)
\end{pmatrix},
\mbox{ and }p=2(y-x).
\end{eqnarray*}
The Dirichlet boundary condition and the right hand side function is set to match the above exact solutions. It is easy to check that the exact solution $\bu$ satisfies the non-homogeneous boundary condition.

For this testing, the WG divergence free algorithm is perform on the triangular grids. The uniform triangular girds are generated by: (1) partition the domain into
$n\times n$ sub-rectangles; (2) divide each square element into two
triangles by the diagonal line with a negative slope. The mesh size
is denoted by $h=1/n.$

For the calculation of the pressure $p_h$, we shall make use of the basis function $\bv\in V_h\backslash D_h$. This basis function is corresponding to the velocity $\bv_b$ related of the normal direction on each edge. Let $\bv\in V_h\backslash D_h$, the pressure $p_h$ is computed as follows,
$$b(\bv,p_h)=a(\bu_h,\bv)-(f,\bv_0).$$

Beside testing two norms of the error in velocity, we also measure the $L^2-$error in pressure. The numerical results in Table \ref{ex2_tri} show an $O(h)$ convergence in the $\3bar\cdot\3bar$ norm for velocity, $O(h^2)$ convergence in the $L^2$-norm for velocity, and $O(h)$ convergence in the $L^2-$norm for pressure, which are confirmed by Theorem \ref{h1-bd}.

\begin{table}[h]
\caption{Test Case 2: Numerical error and convergence rates for the
Stokes equation with non-homogeneous boundary
conditions.}\label{ex2_tri} \center
\begin{tabular}{|c|ccc|}
\hline
$h$ & $\3bar \bu_h-Q_h \bu\3bar$ & $\|\bu_0-Q_0 \bu\|$ &$\|p_h-p\|$ \\
\hline\hline
   2.5000e-01   &2.8901e-01   &4.2990e-02   &2.2624e-01\\ \hline
   1.2500e-01   &1.4367e-01   &1.0896e-02   &1.2246e-01\\ \hline
   6.2500e-02   &7.1997e-02   &2.7432e-03   &6.4525e-02\\ \hline
   3.1250e-02   &3.6052e-02   &6.8773e-04   &3.3224e-02\\ \hline
   1.5625e-02   &1.8037e-02   &1.7210e-04   &1.6871e-02\\ \hline
   7.8125e-03   &9.0202e-03   &4.3038e-05   &8.5037e-03\\ \hline
Conv.Rate       &9.9966e-01   &1.9934       &9.4871e-01\\ \hline
\end{tabular}
\end{table}

\section*{Acknowledgement}
We offer our gratitude to professor Eric Lord and professor David
Singer for their help on obtaining Equation (\ref{key3}).

\vfill\eject

\end{document}